\documentclass[smallextended]{svjour3}       
\smartqed  
\usepackage{graphicx}
\usepackage{amssymb}
\usepackage{amsmath}
\usepackage{times}
\newcommand{\R}{\mathbb{R}}
\begin{document}
\title{Propagation of bursting oscillations in coupled non-homogeneous Hodgkin-Huxley Reaction-Diffusion systems}

\titlerunning{Networks of Hodgkin-Huxley Reaction-Diffusion systems}        

\author{B. Ambrosio, M.A. Aziz-Alaoui and A. Balti}


\institute{B. Ambrosio, M.A. Aziz-Alaoui and A.Balti \at Normandie Univ, UNIHAVRE, LMAH, FR-CNRS-3335, ISCN, 76600 Le Havre, France\\
             \email{benjamin.ambrosio@univ-lehavre.fr}           }


\maketitle

\begin{abstract}
In this paper, we consider networks of reaction-diffusion systems of Hodgkin-Huxley type.
We give a general  mathematical framework, in which we prove existence and unicity of solutions as well as the existence of invariant regions and of the attractor.
Then, we illustrate some relevant numerical examples and exhibit bifurcation phenomena and propagation of bursting oscillations through one and two coupled non-homogeneous systems.
\end{abstract}

\section{Introduction}
Action potential propagation  is a crucial phenomenon for information process  in the nervous system and particularly in the brain.
One of the paradigmatic models, which has been proposed to describe action potential  propagation, is the Hodgkin-Huxley model.
Written initially in 1952, see \cite{HH}, for the description of the electrical activity of the squid giant axon, 
its formalism has served as basis of number of models widely used in Mathematical Neuroscience,
see for example \cite{Ambrosio0,Ambrosio1,Ambrosio2,C1,Fitz,Hin,Mor,Nagumo} and references therein cited.
At that time, Hodgkin and Huxley  used the new voltage clamp technique to maintain constant the membrane potential.
This technique allowed them to elaborate and fit the functional parameters of their model of four ODEs with their experiments. 
Basically, the model is obtained by considering the cell as an electrical circuit. The membrane acts as a capacitor 
whereas ionic currents result from ionic channels acting as variable voltage dependent resistances.
The model takes into account three ionic currents: potassium ($K^+$), sodium ($Na^+$) and leakage (mainly chlorure, $Cl^-$). The  Hodgkin-Huxley (HH) system reads as:
\begin{equation}\label{eq:HH}
\left\{
\begin{array}{rcl}
CV_t  &=&I+\overline{g}_{Na} m^3 h(E_{Na}-V)+\overline{g}_{K} n^4(E_{K}-V)+\overline{g}_{L}(E_{L}-V)\\
n_t  &=& \alpha_{n}(V)(1-n)-\beta_{n}(V)n\\
m_t  &=& \alpha_{m}(V)(1-m)-\beta_{m}(V)m\\
h_t  &=& \alpha_{h}(V)(1-h)-\beta_{h}(V)h,\\
\end{array}
\right.
\end{equation}
%
where subscript $t$ stands for derivative $\frac{d}{dt}$ and
where $I$ is the external membrane current, $C$ is the membrane capacitance, $ \overline{g}_{i},\  E_i, i\in\{K,Na,L\}$ are respectively the maximal 
conductances and the
(shifted) Nernst equilibrium potentials. 
The functions $\alpha (V) $ and $\beta (V)$ describe the transition rates between open and closed states of channels. They read as:
\begin{equation}
\label{eq:param}
\begin{array}{rcl}
 \alpha_{n}(V)=0.01 \frac{-V+10}{\exp{(1-0.1V)}-1},& & \beta_{n}(V) =0.125\exp(-V/80),\\
 \alpha_{m}(V)=0.1 \frac{-V+25}{\exp{(2.5-0.1V)}-1},& & \beta_{m}(V) =4 \exp(-V/18),\\
\alpha_{h}(V)= 0.07\exp(-V/20), & & \beta_{h}(V) = \frac{1}{ 1+\exp( -0.1V+3) )}.
\end{array}
\end{equation}
%

The (shifted) Nernst equilibrium potentials are given by:
$$E_{K}=-12 \;mV,~~E_{Na}=120\; mV,~~E_{L}=10.6 \;mV$$
$$\overline{g}_{K}=36 \;mS/cm^2, \overline{g}_{Na}=120 \;mS/cm^2,\overline{g}_{L}=0.3 \;mS/cm^2.$$
These values are taken from \cite{Izh}, p 37-38, and correspond to those of the Hodgkin-Huxley original paper \cite{HH}, after a change of variables $V=-V$. 
Recall that the Nernst equilibrium potentials are obtained by solving, for each ion $i\in\{K,Na,L\}$ the equation:
\[E_i=\frac{RT}{zF}\ln\frac{[i]_{out}}{[i]_{in}},\]
where $[i]_{in}$ and $[i]_{out}$ are concentrations of the ions inside and outside the cell. $R=8.315$ is the universal gas constant, $T$ is temperature in Kelvin, $F=96, 48$
is the Faraday’s constant, $z$ is the valence of the ion. For example, this computation gives for sodium, with $T=293$, $[i]_{out}=440$, $[i]_{in}=40$ (see \cite{Izh} p 50):
\[E_{Na}\simeq 55,\]
which with a shift of $+65$ gives the value of $120$ used here. Originally, Hodgkin and Huxley used the shift to obtain a potential at rest of approximately $0$.
Before going into more theoretical aspects, we give some interpretation about the form of conductances. The proportion of open potassium channels is $n^4$. 
This comes from the fact that $4$ opening gates of potassium are required to open the potassium channel. Hence, the $n$ gives 
the probability of the gate to be in active state and 
results in the $n^4$ term for potassium. For the sodium, it is supposed that there are there are three gates which open the channels  and one which close them.
Hence, the proportion of 
sodium opened channels is given by $m^3h$, where  it is supposed that $m$ stands for the probability of sodium opening gates to be active 
while $h$ stands for  probability of sodium closing gates
to be active.   For more details on various aspects of the HH model, we refer to \cite{Cro,Erm,Izh} and the original paper \cite{HH}. 
In the present paper, we  consider a general network of reaction-diffusion (RD) HH equations. Note that the study of networks of reaction-diffusion systems start to
attract an increasing number of mathematical analysis, see for example, \cite{Ambrosio1,Ambrosio2,Yang}. They appear naturally in the neuroscience context. 
The general network reads as:
\begin{equation}\label{eq:NnoeudHH-orig}
\left\{
\begin{array}{rcl}
V_{it} &=&dV_{ixx}+ I+\overline{g}_{Na} m_i^3 h(E_{Na}-V_i)+\overline{g}_{K} n_i^4(E_{K}-V_i)+\overline{g}_{L}(E_{L}-V_i)\\
& &+H_i(V_1,...,V_N),\qquad i\in \{1,...,N\},\\
n_{it}  &=& \alpha_{n}(V_i)(1-n_i)-\beta_{n}(V_i)n_i\\
m_{it}  &=& \alpha_{m}(V_i)(1-m_i)-\beta_{m_i}(V)m_i\\
h_{it}  &=& \alpha_{h}(V_i)(1-h_i)-\beta_{h}(V_i)h_i.\\

\end{array}
\right.
\end{equation} 
in a bounded domain $\Omega=(a,b)\subset \R$ and with Neumann boundary conditions $V_{ix}(a)=V_{ix}(b)=0$ and
where,

\begin{equation}
\label{eq:HiNonLin}
H_i(u)=\sum_{j\in\{1,...,N\}}\alpha_{ij}(S-u_i)\Gamma(u_j) ,
\end{equation}
\begin{equation*}
\Gamma(s)=\frac{1}{1+\exp^{-\lambda(s-\theta)}},
\end{equation*}
with 
\begin{equation}
\label{eq:param2}
S=100,\lambda=20, \theta=60,
\end{equation}
and $\alpha_{ij}\geq 0$.\\
The functions $H_i$ represent excitatory nonlinear coupling (i.e. chemical synaptic coupling) between neurons, see \cite{Ambrosio3,Ambrosio4,Bel,C1} and
references therein cited.
Without loss of generality, we have set the constant $C$ equal to $1$. Hence, the value $S=100$ allows the neuron $i$ to receive  an excitatory input each time
one of his
presynaptic neurons crosses the treshold value $\theta=60$. The value $\lambda=20$, implies that  $\Gamma$  approximates the Heaviside function. 
The Neumann condition is chosen here to not induce boundary effects. Here, $I$ and $\alpha's$ are regular non negative bounded function of $x$.
Our first aim is to provide a  mathematical framework for solutions of system \eqref{eq:NnoeudHH-orig}. This is done in section 2. We prove the global existence 
of solutions as well as the existence of invariant regions and the existence of an attractor. 
Then, in the third section, we deal with numerical simulations for a single non-homogeneous HH RD system and two coupled HH RD systems.
In this last part, we exhibit some relevant bifurcation phenomena and propagation of bursting oscillations from first  neuron to the second neuron. 
Note that system \eqref{eq:NnoeudHH-orig} is autonomous, which means that  the bursting oscillations are generated here without external driving component as in 
\cite{Ambrosio0}. See also \cite{Cle}. 
\section{Mathematical Framework}

\subsection{Existence of the solution}
Let $X=C(\bar{\Omega})=C([a,b])$ the space of continuous functions defined on the real interval $[a,b]$. 
We start with the proof of the existence of the solution of \eqref{eq:NnoeudHH-orig}. 
The proof is divided into four parts. First, we prove the local existence and uniqueness of a mild-solution, thanks to a fixed point theorem. Then, we prove the regularity 
of this mild solution. Finally, we prove the bounds for $n_i,m_i,h_i$ as well as the global existence. 
\begin{theorem}
For all initial conditions $(V_i(0),n_i(0),m_i(0),h_i(0))\in X^{4N}$ such that $\forall x \in [a,b]$, $n_i(0,x),m_i(0,x),h_i(0,x) \in [0,1]$,
there exists a unique solution $U \in C([0,+\infty[,X^{4N})$ of \eqref{eq:NnoeudHH-orig}. Furthermore, $U \in 
C^1(]0,+\infty[,X^{N})$ and for all $t>0$, $n_i(t,x),m_i(t,x),h_i(t,x) \in [0,1]$.
\end{theorem}
  \begin{proof}
It is  known that the equation
\[u_t=u_{xx}\]
with 
\[u_x(a)=u_x(b)=0\]
generates  an analytical semigroup on $X$, see \cite{Lunardi}, Theorem 3.1.22 p 98.
 Let us denote by $S(t)_{t\geq0}$ this semigroup. 
  The first step is to look for mild solutions of   \eqref{eq:NnoeudHH-orig}. 
  Mild-solutions are solutions of the variation of constants formula of \eqref{eq:NnoeudHH-orig}, and written with the semigroup notation. 
  The general network equation for mild-solutions reads as:
\begin{equation}\label{eq:NnoeudHH-mild}
\left\{
\begin{array}{rcl}
V_i(t)  &=&\displaystyle  S(t)V_i(0)+\int_0^tS(t-s)(F_{V}(V_i(s),n_i(s),m_i(s),h_i(s))\\
& &+H_i(V_1,...,V_N))ds\\
n_i(t) &=&\displaystyle n_i(0)+\int_0^tF_{n}(V_i(s),n_i(s))ds\\
m_i(t) &=&\displaystyle m_i(0)+\int_0^tF_{m}(V_i(s),m_i(s))ds\\
h_i(t) &=&\displaystyle h_i(0)+\int_0^tF_{h}(V_i(s),h_i(s))ds,\\
\end{array}
\right.
\end{equation} 
where:
\begin{equation}
\label{eq:F_VF_n}
\begin{array}{rcl}
F_V(V,n,m,h)&=& I+\overline{g}_{Na} m^3 h(E_{Na}-V)+\overline{g}_{K} n^4(E_{K}-V)+\overline{g}_{L}(E_{L}-V),\\
F_n(V,n) &=&\alpha_{n}(V)(1-n)-\beta_{n}(V)n,\\
F_m(V,n)&=&   \alpha_{m}(V)(1-m)-\beta_{m}(V)m,\\
F_h(V,n)&=&\alpha_{h}(V)(1-h)-\beta_{h}(V)h.\\

  \end{array}
  \end{equation}
The local existence  follows by application of the fixed point theorem, to the function:
\begin{equation*}
\Phi_{U_0}:C([0,T],X^{4N})\rightarrow C([0,T],X^{4N})
\end{equation*} 
defined by the right-hand side of \eqref{eq:NnoeudHH-mild}, in which we fix the initial condition $U_0 \in X^{4N}$ (below, for sake of simplicity, 
we drop the subscript $U_0$ on $\phi$).
Let $U$ denote the function $(V_i,n_i,m_i,h_i)(t),i \in\{1,...,N\}, t \in [0,T]$. For all $U\in B(U_0,r)$, the ball of center $U_0$ and radius $r$ in $C([0,T],X^{4N})$,
we have, for fixed $T>0$
\[||\Phi(U)-U_0||\leq KT,\]
where $||\cdot||$ stands here for the sup norm on $C([0,T],X^{4N})$ and where $K$ is a constant depending on $r$ which comes from the boundedness theorem.  
We have also used the property of the linear semigroup $S(t)$: $||S(t)u||_X\leq ||u||_X$ and $\lim_{t\rightarrow 0} S(t)u_0=u_0$, see \cite{Lunardi} p 35. 
It follows that, for $T<\frac{r}{K}$, $\Phi(U)\in B(U_0,r)$.
By analog computations, we prove that
\[||\Phi(U_1)-\Phi(U_2)||\leq K_2T||U_1-U_2||,\]
which proves that for $T<\min(\frac{1}{K_1},\frac{1}{K_2})$, $\phi$ is a contraction mapping.
Therefore, $\Phi$ has  a unique fixed point which provides the existence and uniqueness of the local mild  solution.\\
Now, we deal with the regularity. We already know, that, solutions belong to $C([0,T],X^{4N})$. 
Since the $H$ term, does not complicate the proof, without loss of generality, we drop the subscript $i$.
A computation shows that the quantity
\begin{equation*}
\frac{V(t+h)-V(t)}{h}
\end{equation*}
is well defined, for $t>0$, as $h \rightarrow 0$ and that the resulting function is contiuous on $X$. For $t=0$, the derivative may not exist,
since $\frac{S(h)V_0-V_0}{h}$ admits a limit as $h\rightarrow 0$ if and only if $v_0\in C''([a,b])$. For functions $m,n$ and $h$ the same computation shows that they 
belong to $C^1([0,T],X)$. Hence, the regularity statement follows.\\
Now, we deal with  the global existence.
First, we prove that $n_i, m_i$ and $h_i$  remain in $[0,1]$. Without loss of generality, we drop the $i$ subscript and consider only the $n$ variable.\\
It satisfies,
\[n_t=\alpha(V)(1-n)-\beta(V)n\]
Let
\[A(t)=\int_{0}^{t} \alpha(V) ds,~~B(t)=\int_{0}^{t} \beta(V) ds.\]
Then,
\[n(t)=\exp(-A(t)-B(t))\big[n_0+\int_{0}^{t}\alpha(V(s))\exp(A(s)+B(s))  ds\big]\]
with $0 \leq  \alpha(V)$ and $ 0\leq \beta(V) $.\\
It follows that:
\[ n(t)\geq 0.\]
 Also,
 \begin{equation*}
 \begin{array}{rcl}
   n(t)& \leq & \displaystyle \exp(-A(t)-B(t))\big[n_0+\displaystyle \int_{0}^{t}( \alpha(V) +\beta(V))\exp{(A(s)+B(s)})ds\big]\\
    &=&\exp({-A(t)-B(t)})\big[n_0+(\exp{(A(t)+B(t))}-1)ds\big]\\
   &=& 1+(n_0-1)\exp(-A(t)-B(t))\\
   &\leq& 1.
 \end{array}
 \end{equation*}

Then, the global existence follows from the inequality,
\begin{equation*}
||V_i(t)||_X\leq ||V_i(0)||_X+\int_0^t(C_1||V_i(s)||_X+C_2)ds
\end{equation*}
which comes from the first equation of \eqref{eq:NnoeudHH-orig} and
which by Gronwall inequality yields:
\[\||V_i(t)||_X \leq ||V_i(0)||_X \exp{(C_1t)}+\frac{C_2}{C_1}.\]

\end{proof}

\subsection{Existence of an invariant region}
We assume that the values of parameters are as specified in \eqref{eq:param} and \eqref{eq:param2}. 
This implies in particular  that are $E_K<E_L<S<E_{Na}$. We also assume that $\bar{g}_L(E_{Na}-E_L)>\sup_{x\in [a,b]}I(x)$\footnote{We can drop this last asumption. 
In this case, the upper bound $E_{Na}$ of the invariant region for $V_i$ is replaced by $\frac{1}{\bar{g}_L}\sup_{x\in [a,b]}I(x)+E_L$.}. 
\begin{theorem}
We assume that $\forall i \in \{1,...,N\}$, $ \forall x\in [a,b]$ $V_i(0,x) \in  [E_K,E_{Na}]$ and $n_i(0,x),m_i(0,x),h_i(0,x) \in [0,1]$, then $\forall t>0$ and  $\forall x\in (a,b)$,
\[V_i(t,x) \in  [E_K,E_{Na}] \mbox{ and } n_i(t,x),m_i(t,x),h_i(t,x) \in [0,1].\]
\end{theorem}
\begin{proof}
We have already  proved that $n_i(t,x),m_i(t,x),h_i(t,x) \in [0,1]$.
To prove $V_i(t,x) \in  [E_K,E_{Na}]$ we use the following argument. 
Let $t_0$ the first time that $V_i$ reaches the value $E_K$ at the point $x_0$, then $\frac{\partial^2 V_i}{\partial x^2}\geq 0$ 
which implies $V_{it}(t_0,x_0)>0$ since all the other terms are non negative and $\bar{g}_L(E_L-E_K)>0$.
Analogously, let $t_0$ the first time that $V_i$ reaches the value $E_{Na}$ at the point $x_0$ then $\frac{\partial^2 V_i}{\partial x^2}\leq 0$ 
which implies $V_{it}(t_0,x_0)<0$  since all the other terms are non positive and $\bar{g}_L(E_L-E_{Na})<0$. 
\end{proof}
\subsection{Existence of an attractor}
We denote by $K=C([a,b],([E_K,E_{Na}]\times[0,1]^3)^N)$. In this section, we use the framework of \cite{Tem} to establish the existence of an invariant, compact and connected 
set $\mathcal{A}\subset K$ which attracts all tre trajectories starting in $K$. We first recall the definition of the $\omega-$limit set. 
We denote by $(T(t))_{t\geq 0}$ the semi-group associated with equation \eqref{eq:NnoeudHH-orig}.
\begin{definition}
For $U_0 \in K$, we define the $\omega-limit$ set of $U_0$ by:
\[\omega(U_0)=\bigcap_{t>0}\overline{\bigcup_{s \geq t}T(s)U_0},\]
and similarly the $\omega-limit$ set of $K$ by:
\[\omega(K)=\bigcap_{t>0}\overline{\bigcup_{s \geq t}T(s)K}.\]
\end{definition}
The following theorem holds.
\begin{theorem}
\label{th:comp}
Let 
\[\mathcal{A}=\omega(K).\]
Then $\mathcal{A}$ is a nonempty invariant, compact and connected set in $X^{4N}$. Furthermore, for all initial condition $U_0$ in $K$,
the solution $T(t)U_0$ of  \eqref{eq:1noeudHH-orig} starting at $U_0$ verifies:
\[ \lim_{t\rightarrow +\infty}\inf_{z\in \mathcal{A}}||T(t)U_0-z||_X=0.\] 
\end{theorem}
\begin{proof}
We use the framework of \cite{Tem}. Hence, we need to establish some compacity on trajectories. For convenience, we drop the subscript $i$. For variable $V$,
without loss of generality, we write
\begin{equation}
V(t)=T_1(t)U_0=S(t-t_0)V(t_0)+\int_{t_0}^tS(t-s)F_V(m,n,h)ds,
\label{eq:V}
\end{equation}
since the coupling term $H$ does not change the proof arguments.
We prove that $\overline{\cup_{t\geq t_0} T_1(t)K}$, is compact in $X$, for some $t_0>0$. By Ascoli's theorem, it is sufficient to prove the equicontinuity  of 
\[\bigcup_{t \geq t_0}T_1(s)K\] 
 for some $t_0>0$. 
 Then, we deal with variable $n$. For, variables $m,h$ the same arguments apply. We prove that:  
 \[n(t)=T_2(t)U_0=T^1_2(t)U_0+T^2_2(t)U_0,\]
 where 
 \[\lim_{t\rightarrow +\infty}\sup_{s\geq t,U_0\in K}||T^1_2(s)U_0||_X=0,\]
 and
 \[\overline{\cup_{t\geq t_0} T_2^2(t)K} \mbox{ is compact for some } t_0>0.\]
 According to \cite{Tem}, this suffices to prove theorem \ref{th:comp}.
 Let us start with $V$. Recall that for all $t>0$, $V(t)\in C^\infty([a,b])$. Now, we establish a bound on $V_x$. For all $u\in C^2([a,b])$ it holds,
\[u'=\int_a^xu''(y)dy+u'(a)=\int_a^xu''(y)dy,\]
 since we assume $u'(a)=0$. It follows that
\[||u'||_X\leq (b-a)||u"||_X,\]
and that:
\begin{equation}
||u||_{2,p}\leq C(||u||_p+||u"||_p),
\label{AgDoNi}
\end{equation}
with $C=(1+(b-a)^{p+1})^{\frac{1}{p}}$, where $||\cdot||_{k,p}$ denotes the norm in the Sobolev space $W^{k,p}$ and $||\cdot||_p$ denotes the norm in the Lebesgue space $L^{p}$. Note that inequality \eqref{AgDoNi}, is actually a particular case of  Agmon-Douglis-Niremberg estimates, 
see \cite{Lunardi} p 72 and \cite{AgDoNi}. However, since we consider here the 1-dimensional case the proof is much more simpler. 
Next, we adapt some arguments from \cite{Har}. First, note that the following classical inequalities hold:
\begin{equation*}
||S(t)u_0||_p\leq C_p||u_0||_p, \forall t \in ]0,1],
\end{equation*}
and
\begin{equation*}
||AS(t)u_0||_p\leq \frac{C_p}{t}||u_0||_p, \forall t \in ]0,1],
\end{equation*}
for some generic constants $C_p$, see  \cite{Lunardi} p 35. This implies that:
\begin{equation*}
||S(t)u_0||_{2,p}\leq C(1+\frac{1}{t})||u_0||_p, \forall t \in ]0,1],
\end{equation*}
Now, let $\theta \in ]0,1[$, $p$ and $r$ such that:
\begin{equation}
1<\theta(2-\frac{1}{p})-\frac{1-\theta}{r},
\end{equation}
for example, we can choose $\theta=\frac{3}{4}$, $p>3$ and $r>3$. Then, we have
,see \cite{Har} and also \cite{Henry},
\begin{equation*}
||u||_{1,\infty}\leq C(\theta)||u||^\theta_{2,p}||u||_r^{1-\theta}.
\end{equation*}
Finally, we obtain that for all $t\in ]0,1]$:
\begin{equation*}
||S(t)u_0||_{1,\infty}\leq C(1+\frac{1}{t})^\theta||u_0||_p^\theta||u_0||_r^{1-\theta}.
\end{equation*}
Applying this result in \eqref{eq:V} implies:
\begin{equation}
\label{eq:boundvx}
||V_x(t)||_X \leq C_1||V(t_0)||_X+C_2\int_{t_0}^t\frac{1}{(t-s)^\theta}ds,
\end{equation}
where $C_2$ comes from the fact that $V,m,n,h$ remain in $[E_K,E_{Na}]\times[0,1]^3$. This implies that
\[||V_x(t)||_X \leq C ,\forall t \in [t_0,t_0+1],\]
where the constant $C$ depends on $K$ but not on $t_0$, neither in $U_0$. This proves the equicontinuity of  
\[\bigcup_{t \geq t_0}\bigcup_{U_0\in K}V(t).\] 
Now, we deal with $n$. We compute,
\begin{equation*}
|n(x,t)-n(y,t)|,
\end{equation*}
we obtain (for convenience, we only write the dependence on $x$ and $y$),
\begin{equation*}
\frac{\partial}{\partial t}(n(x)-n(y))=-(\alpha(y)+\beta(y))(n(x)-n(y))+\big((1-n(x))(\alpha(x)-\alpha(y))+n(x)(\beta(y)-\beta(x))\big) 
\end{equation*}
which by integration leads to 
\begin{equation*}
|n(x,t)-n(y,t)|\leq \exp(-\xi t)|n(x,0)-n(y,0)|+\frac{C}{\xi}|x-y|,
\end{equation*}
where $\xi=\inf_{y\in [a,b]}(\alpha(y)+\beta(y))>0$ and $C$ is a generic constant. This gives the announced statement.
\end{proof}
\begin{remark}
The difficult point in the proof was to obtain the bound \eqref{eq:boundvx}.
\end{remark}

%

\section{Bifurcations and propagation of oscillations in one and two coupled non-homogeneous neurons}
\subsection{Bifurcation in a single non-homogeneous neuron}
In this section, we deal with equation  \eqref{eq:NnoeudHH-orig} with $N=1$. It reads:
\begin{equation}\label{eq:1noeudHH-orig}
\left\{
\begin{array}{l}
V_t  =dV_{xx}+ I(x)+\overline{g}_{Na} m^3 h(E_{Na}-V)+\overline{g}_{K} n^4(E_{K}-V)+\overline{g}_{L}(E_{L}-V)],\\
\\
n_t  = \alpha_{n}(V)(1-n)-\beta_{n}(V)n, \\
\\
m_t  = \alpha_{m}(V)(1-m)-\beta_{m}(V)m,\\
h_t  = \alpha_{h}(V)(1-h)-\beta_{h}(V)h.\\

\end{array}
\right.
\end{equation} 
We have proceeded to numerical simulations in the domain $\Omega=(a,b)=(0,100), d=1$, on time interval $[0,500]$,
using our own C++ program with a uniform finite difference scheme in space and a Runge-Kutta method
integration in time. Hence, the space-mesh reads, $(ih)_{i \in \{0,100\}}, h=1$, and the time-mesh  reads, $(n\Delta t)_{n \in \{0,50000\}}, \Delta t=0.01.$ Note that since the equation 
for the potential is linear in $V$ and the equation for channels are linear in $n,m,h$. One can, from step to step, compute $n(t+\Delta t),m(t+\Delta t),h(t+\Delta t)$  for fixed $V(t)$. And use a similar argument for $V$ computation. 
A discussion on this aspects can be found in \cite{Bal}.
The particularity is that we choose a regular non homogeneous $I$, approximating the discontinuous function $\tilde{I}$ defined by:
\begin{equation}\label{eq:I}
\tilde{I}(x)=\left\{
\begin{array}{rl}
I_0  &\mbox{ if }  x<\frac{b-a}{10}\\
0& \mbox{ otherwise }
\end{array}
\right.
\end{equation}
In the following descriptions, for simplicity we identify $\tilde{I}(x)$ with $I(x)$. We will use the same convention for $\alpha_{21}$ on section 3.2. 

 We emphasize here three different regimes corresponding to different region of parameter value $I_0$.
\subsubsection*{Stationary and periodic solutions for $I_0\simeq 5.2$}
\textbf{Convergence toward stationary solution for $I_0=5.2$}\\
For several initial conditions and for  small $I_0$, i.e. for $0\leq I_0\leq 5.2$, 
we obtain asymptotically a convergence toward a steady state. In figure \ref{fig2}, we have illustrated this phenomenon for $I_0=5.2$, for initial condition $(1,1,1,1)$. 
The first panel illustrates the stationary space dependent solution, which is attained asymptotically. 
The second panel illustrates the time evolution for $x=0$ and $x=100$.\\
\textbf{Convergence toward periodic solution for $I_0=5.3$}\\
For the same initial conditions as before and for  $I_0=5.3$, we obtain asymptotically a convergence toward a periodic solution.  
This periodic solution consists of large solutions starting at $x=0$ and propagating form left to right. 
We have illustrated this phenomenon in figure \ref{fig3-a}, for initial condition $(1,1,1,1)$. 
The first panel illustrates the propagation of oscillations from left to right. In fact, it shows the solution along the space for a fixed time $t=500$. 
As in the previous figures, the second panel illustrates the time evolution for $x=0$ and $x=100$. It shows periodic solutions.\\
\textbf{Convergence toward stationary solution for $I_0=5.3$}\\
We note for this value the coexistence of attractive stationary and periodic solutions. Indeed, for $I_0=5.3$, and initial condition $(0,0,0,0)$, we observe 
a convergence toward a stationary solution. The phenomenon is depicted as before in figure \ref{fig3-b}.
\subsubsection*{Propagation of busting oscillations}

\textbf{Convergence toward a solution with propagation of busting oscillations for $I_0=130$}\\
For  $I_0=130$, we observe  propagation of busting oscillations. The left cells oscillate continuously, but not all the oscillations propagate trough the right-cells. In fact, the right-cells alternate quiescent phases with oscillating phases, which is a bursting phenomenon. We illustrate this phenomena in figure \ref{fig4}. The left-top panel illustrates the time evolution at $x=0$. We observe  oscillations. The right-top panel illustrates the time evolution at $x=100$. We observe bursting phenomena. The left-middle panel illustrates the time evolution at $x=8$. We observe the so-called mixed-mode oscillations: an alternation of small and large oscillations. The right-middle panel illustrates the time evolution at $x=8$, in the (V,n,m) variables. The left-down panel illustrates the propagation of oscillations in space at fixed time $t=200$.  The right-down panel illustrate the quiescent phase in space at fixed time $t=250$.\\
\subsubsection*{Death-spot}
\textbf{Convergence toward a solution with death-spot $I_0=145$}\\
For  $I_0=145$, we observe this another interesting phenomenon. The left cells oscillate continuously, but the wave propagation fails to propagate along the whole space. This is known as the death-spot phenomena, see \cite{Ambrosio0,Yan1,Yan2} and  references therein cited. The right cells seem to have reached a stationary state whereas the left cells are still oscillating. The left-top panel illustrates the time evolution at $x=0$. We observe  oscillations. The right-top panel illustrates the time evolution at $x=100$. We observe a near-stationary solution. The left-middle panel illustrates the time evolution at $x=8$. We observe small  oscillations. The right-middle panel illustrates the time evolution at $x=8$, in the (V,n,m) variables. The left-down panel illustrates the oscillations of left cells and the near stationary state for right cells in space at fixed time $t=200$ and time $t=250$.

\begin{figure}[ht]
\begin{center}
\includegraphics[scale=0.3]{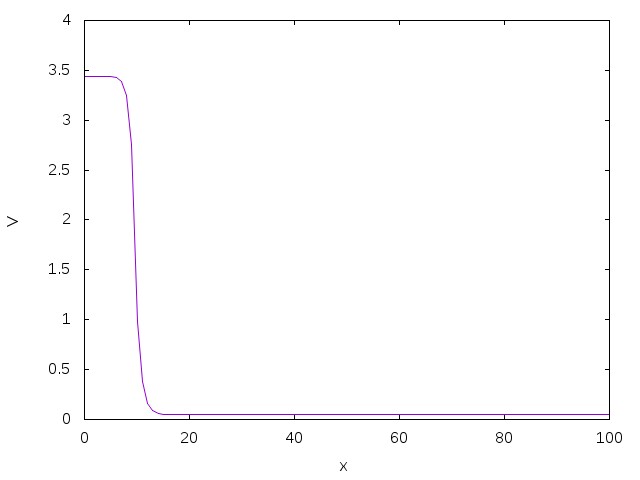}
\includegraphics[scale=0.3]{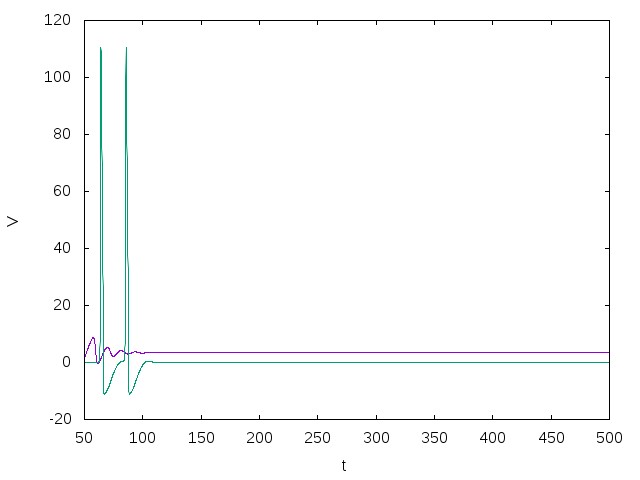}
\caption{Evolution toward a stationary solution for $I_0=5.2$ and initial condition $(1,1,1,1)$. The left panel shows $V(t,x)$ for fixed $t=500$.
The right panel shows time evolution of $V(t,x)$ for $x=0$ and $x=100$.} 
\label{fig2}
\end{center}
\end{figure}
\begin{figure}[ht]
\begin{center}
\includegraphics[scale=0.3]{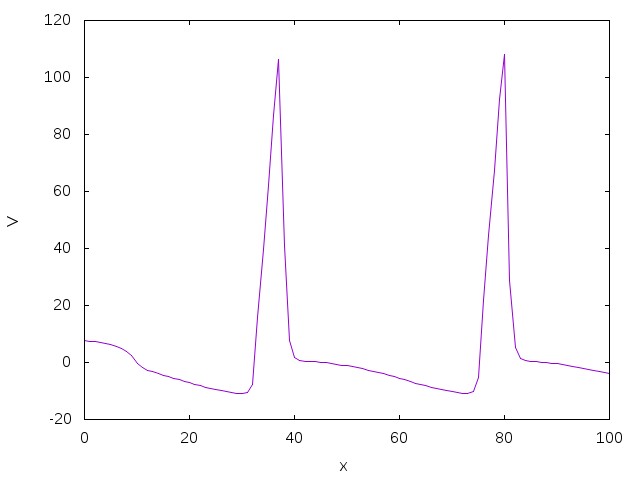}
\includegraphics[scale=0.3]{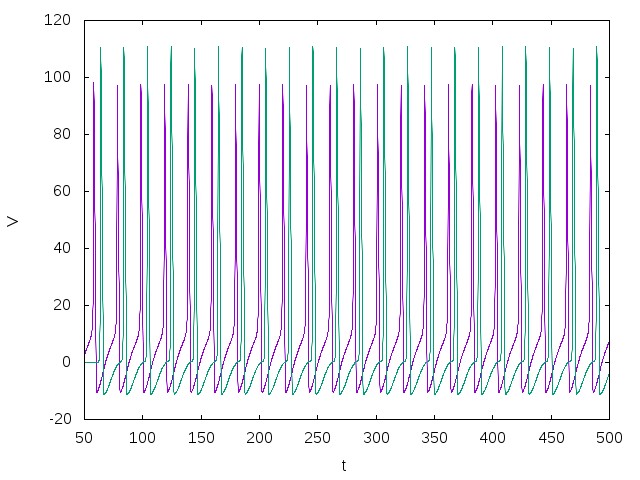}
\caption{Evolution toward a wave propagation solution for $I_0=5.3$ and initial condition $(1,1,1,1)$. The left panel shows $V(t,x)$ for fixed $t=500$. 
The right panel shows time evolution of $V(t,x)$ for $x=0$ and $x=100$.} 
\label{fig3-a}
\end{center}
\end{figure}
\begin{figure}[ht]
\begin{center}
\includegraphics[scale=0.3]{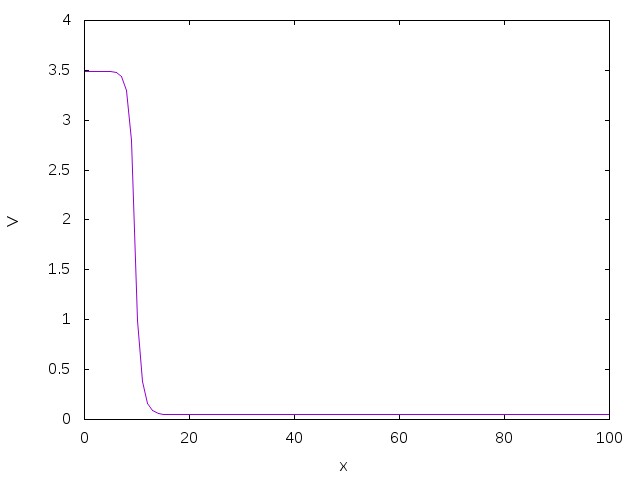}
\includegraphics[scale=0.3]{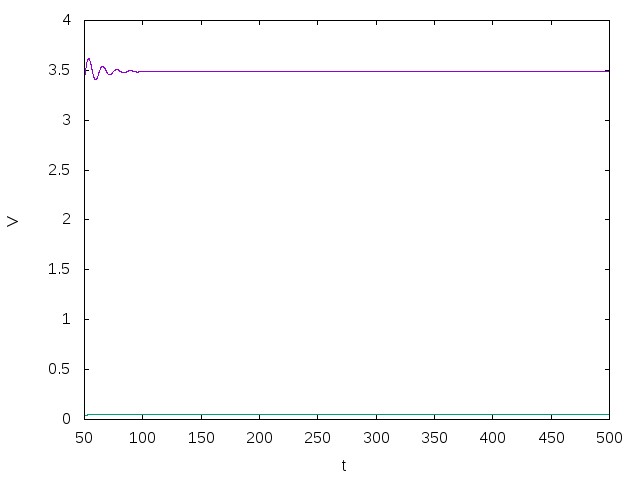}
\caption{Evolution toward a stationary solution for $I_0=5.3$ and initial condition $(0,0,0,0)$. The left panel shows $V(t,x)$ for fixed $t=500$.
The right panel shows time evolution of $V(t,x)$ for $x=0$ and $x=100$.} 
\label{fig3-b}
\end{center}
\end{figure}

\begin{figure}[ht]
\begin{center}
\includegraphics[scale=0.3]{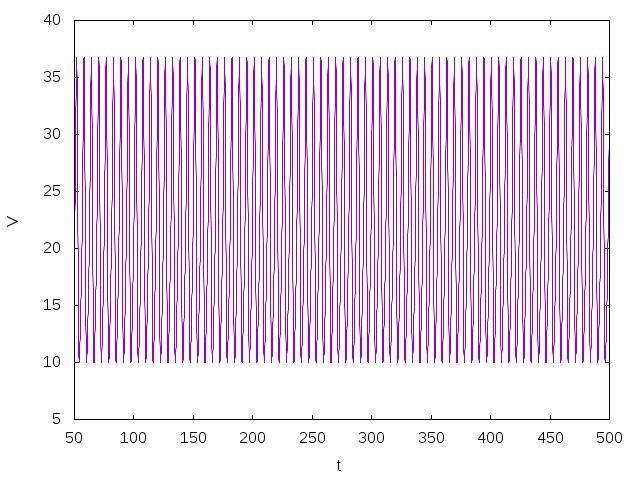}
\includegraphics[scale=0.3]{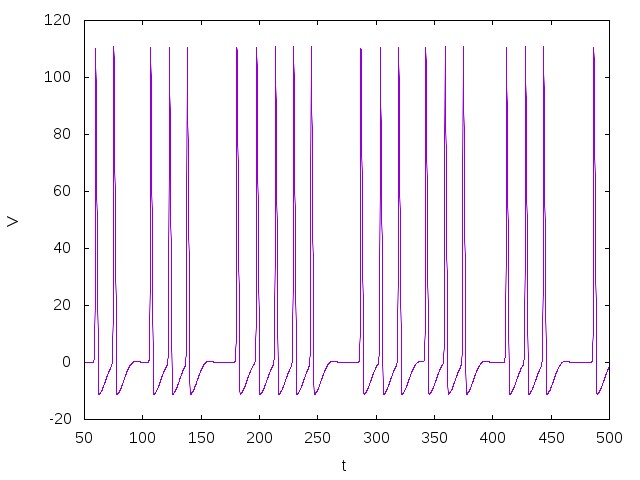}\\
\includegraphics[scale=0.3]{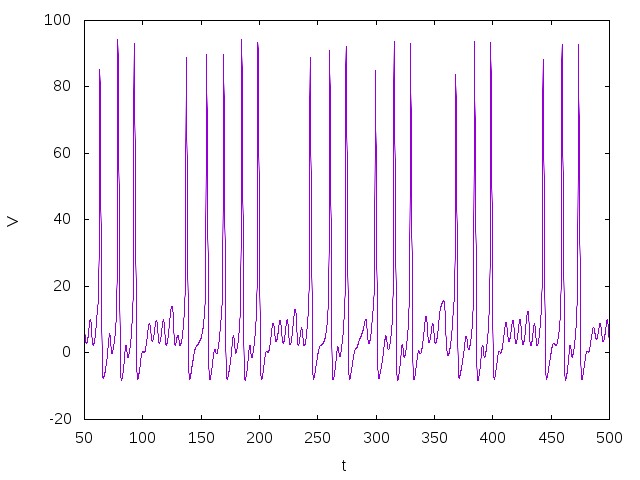}
\includegraphics[scale=0.3]{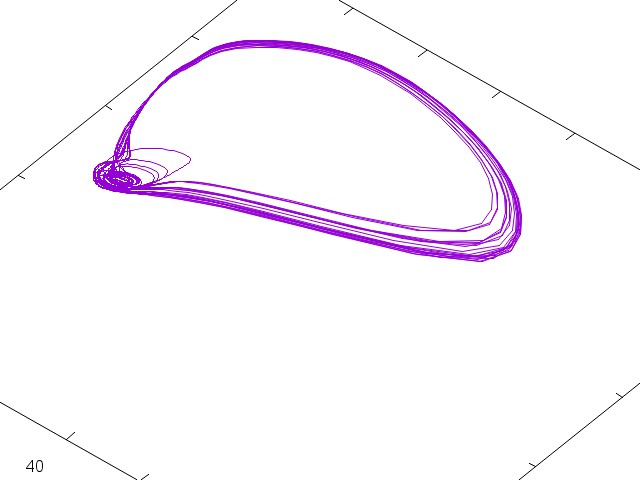}\\
\includegraphics[scale=0.3]{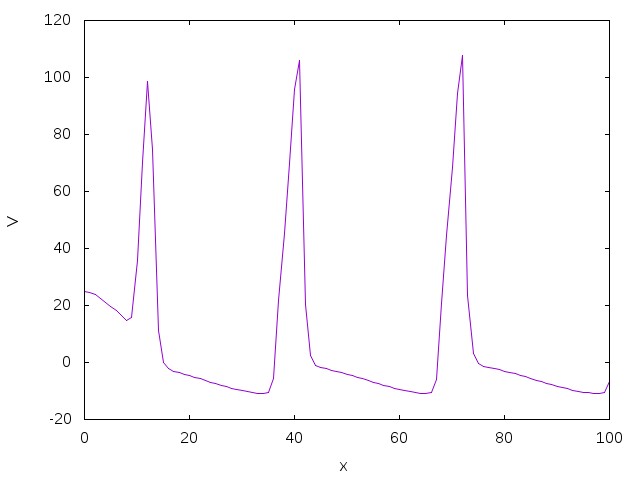}
\includegraphics[scale=0.3]{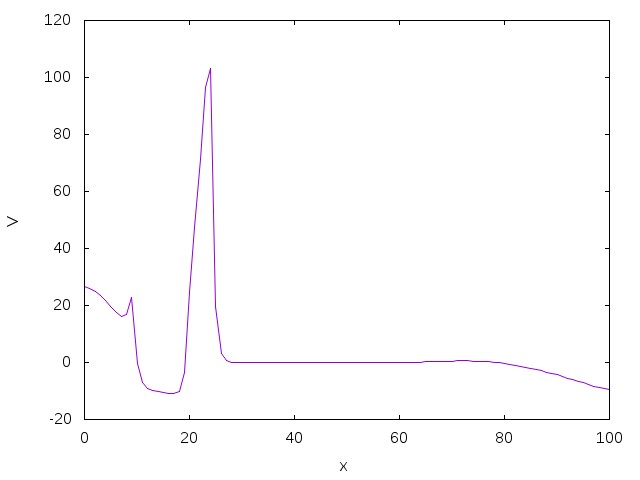}
\caption{Propagation of bursting oscillations for $I_0=130$. The left-top panel shows $V(t,0)$ for $t\in[50,500]$. 
The right-top panel shows $V(t,100)$ for $t\in[50,500]$.The left-middle panel shows $V(t,8)$  for $t\in[50,500]$. 
The right-middle panel shows time evolution of $(V,n,m)(t,8)$  for $t\in[50,500]$. The left-down panel shows $V(t,x)$ for fixed $t=200$.
The right-down panel shows time evolution of $V(t,x)$ for  fixed $t=250$.} 
\label{fig4}
\end{center}
\end{figure}

\begin{figure}[ht]
\begin{center}
\includegraphics[scale=0.3]{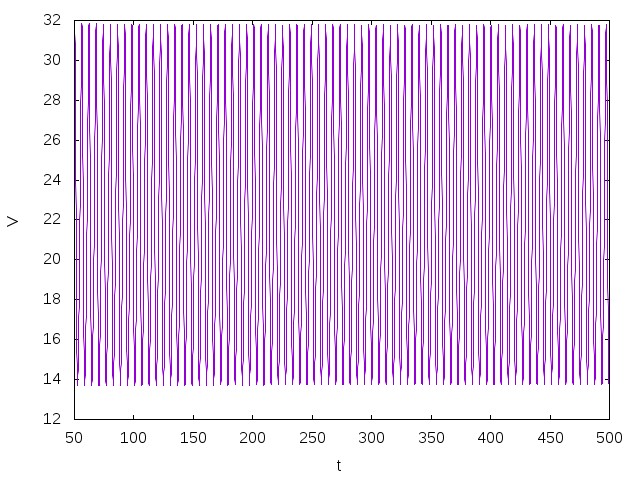}
\includegraphics[scale=0.3]{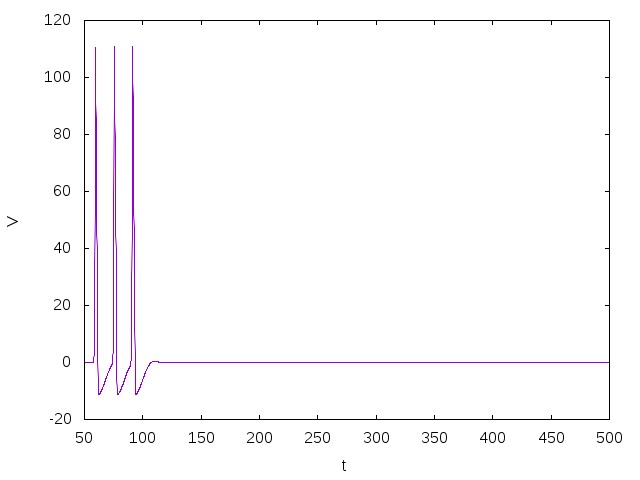}\\
\includegraphics[scale=0.3]{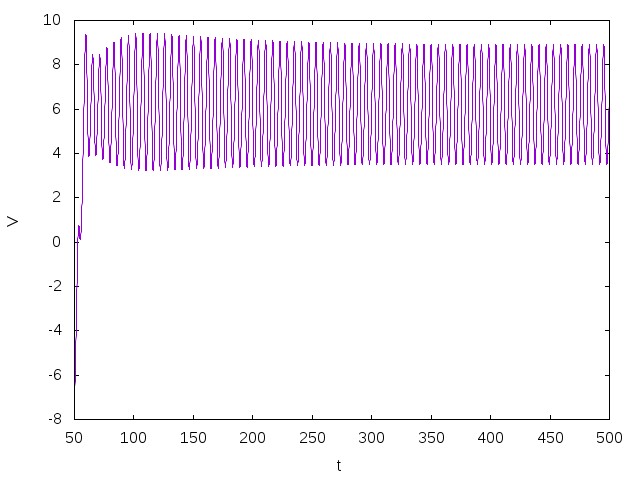}
\includegraphics[scale=0.3]{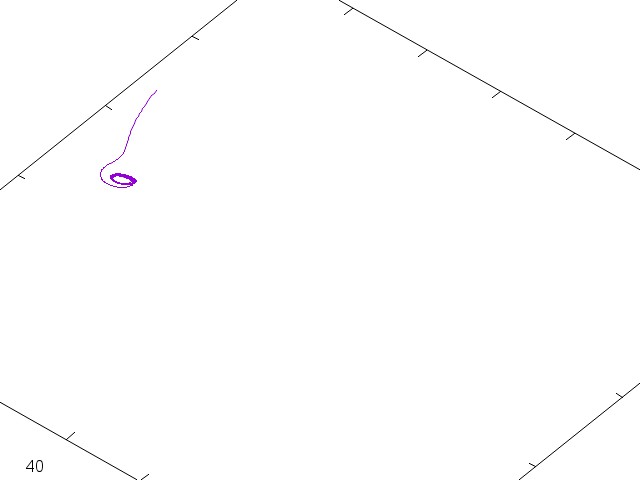}\\
\includegraphics[scale=0.3]{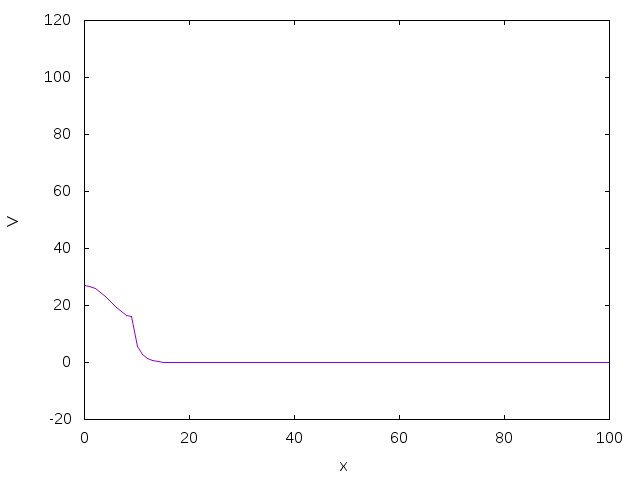}
\includegraphics[scale=0.3]{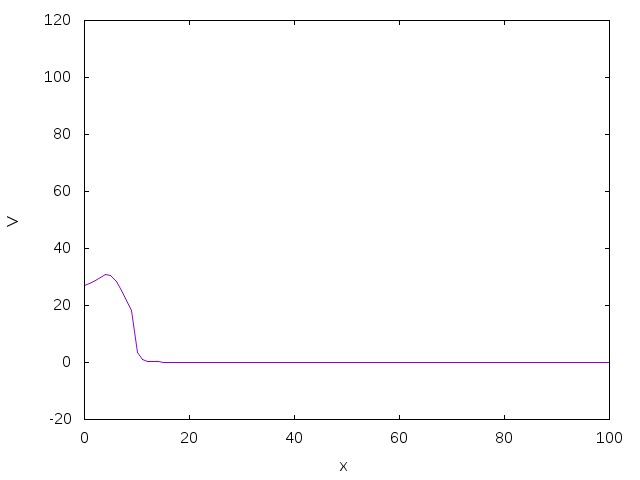}
\caption{Death spot phenomena for $I_0=145$. The left-top panel shows $V(t,0)$ for $t\in[50,500]$. The right-top panel shows $V(t,100)$ for $t\in[50,500]$. 
The left-middle panel shows $V(t,10)$  for $t\in[50,500]$. The right-middle panel shows time evolution of $(V,n,m)(t,10)$  for $t\in[50,500]$.
The left-down panel shows $V(t,x)$ for fixed $t=200$. The right-down panel shows time evolution of $V(t,x)$ for fixed $t=250$.} 
\label{fig5}
\end{center}
\end{figure}

\subsection{Propagation of  bursting oscillations in 2 coupled neurons}
In this section, we deal with equation  \eqref{eq:NnoeudHH-orig} with  $N=2$, $I_2=0$, $\alpha_{ij}=0$ for $(i,j)\neq (2,1)$ and,
\begin{equation}\label{eq:al}
\alpha_{21}(x)=\left\{
\begin{array}{rl}
0  &\mbox{ if }  x<\frac{9}{10}(b-a),\\
1& \mbox{ otherwise. }
\end{array}
\right.
\end{equation}
 It reads:
\begin{equation}\label{eq:2noeudHH-orig}
\left\{
\begin{array}{rcl}
V_{1t}  &=&dV_{1xx}+ I_1(x)+\overline{g}_{Na} m_1^3 h_1(E_{Na}-V_1)+\overline{g}_{K} n_1^4(E_{K}-V_1)+\overline{g}_{L}(E_{L}-V_1),\\
n_{1t}  &=& \alpha_{n}(V_1)(1-n_1)-\beta_{n}(V_1)n_1, \\
m_{1t}  &=& \alpha_{m}(V_1)(1-m_1)-\beta_{m}(V_1)m_1,\\
h_{1t}  &=& \alpha_{h}(V_1)(1-h_1)-\beta_{h}(V_1)h_1,\\
V_{2t}  &=&dV_{2xx}+ I_2(x)+\overline{g}_{Na} m_2^3 h_2(E_{Na}-V_2)+\overline{g}_{K} n_2^4(VE_{K}-V_2)+\overline{g}_{L}(E_{L}-V_2)\\
& &+\alpha_{21}(x)(S-V_2)\Gamma(V_1),\\
n_{2t}  &=& \alpha_{n}(V_2)(1-n_2)-\beta_{n}(V_2)n_2, \\
m_{2t}  &=& \alpha_{m}(V_2)(1-m_2)-\beta_{m}(V_1)m_2,\\
h_{2t}  &=& \alpha_{h}(V_2)(1-h_2)-\beta_{h}(V_2)h_2.\\

\end{array}
\right.
\end{equation} 
We have proceeded to numerical simulations as before with
\begin{equation}\label{eq:I2}
I(x)=\left\{
\begin{array}{rl}
130  &\mbox{ if }  x<\frac{b-a}{10}\\
0& \mbox{ otherwise }
\end{array}
\right.
\end{equation}
As the neuron $1$ is not affected by the coupling, the dynamic of this neuron has already been described. For the second neuron, we observe the propagation of bursting oscillations from neuron $1$ to neuron $2$. As the coupling acts only on the right side of the neuron $2$, oscillations propagate from right to left. It is interesting to note that there is also a short propagation toward the right boundary. We can observe this on the spike with a double head at the right side on the right-down panel of figure 6.

\section{Conclusion}
In this paper, we have considered a general network of HH RD systems. We have proved existence and uniqueness of solutions as well as the existence
of  invariant region and of the attractor in the space of continuous functions. We have also exhibited bifurcation phenomena and propagation of bursting oscillations 
along one and two coupled neurons.
These numerical results show that a rich behavior may be obtained thanks to the space-inhomogeneities. In future works, we aim to develop theoretical tools to describe and analyze
these phenomena.

\section*{Acknowledgments}
The authors would like to thank R\'egion Normandie and 
FEDER XTERM   for financial support.

\begin{figure}[ht]
\begin{center}
\includegraphics[scale=0.3]{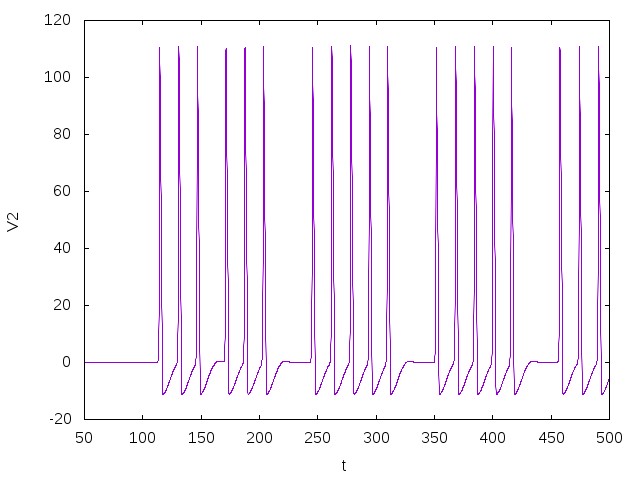}
\includegraphics[scale=0.3]{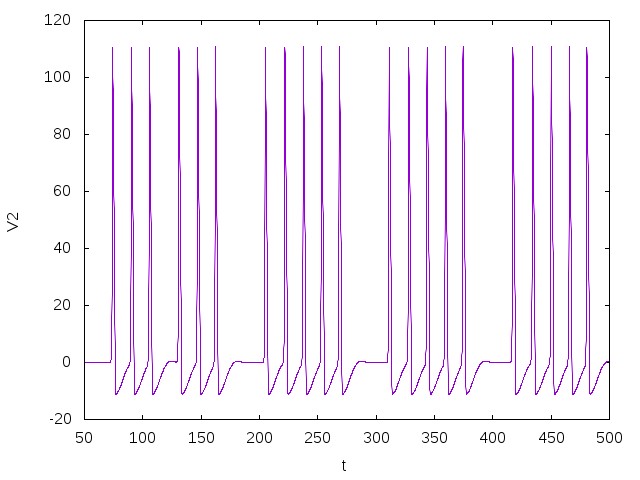}\\
\includegraphics[scale=0.3]{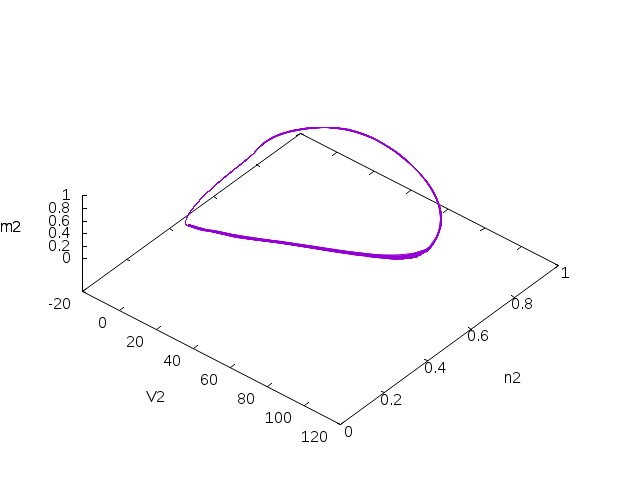}\\
\includegraphics[scale=0.3]{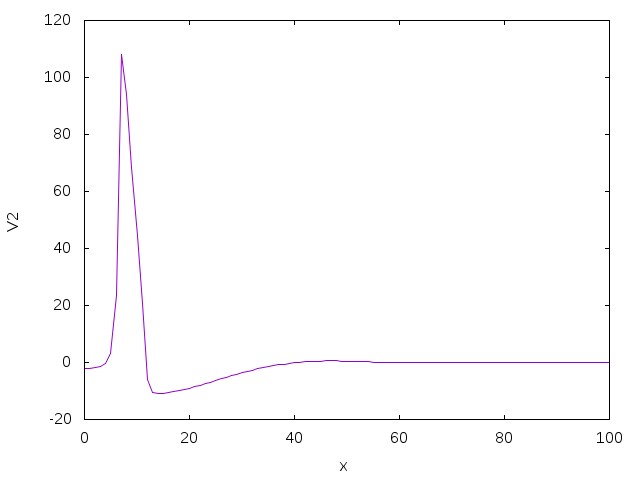}
\includegraphics[scale=0.3]{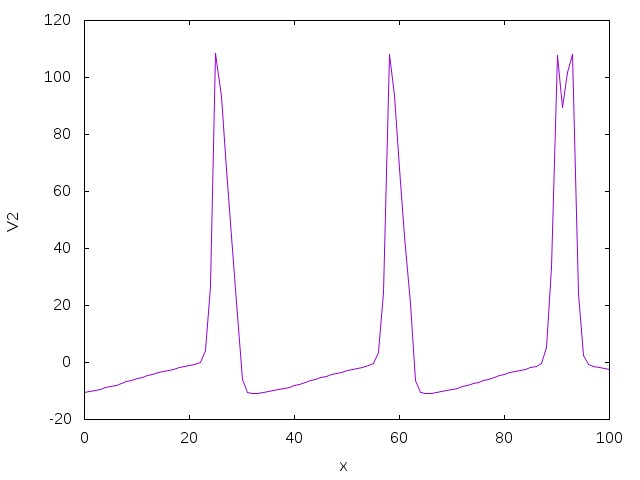}
\caption{Propagation of bursting oscillations in the second neuron.  The left-top panel shows $V(t,0)$ for $t\in[50,500]$. 
The right-top panel shows $V(t,100)$ for $t\in[50,500]$. The middle panel shows $(V,n,m)(t,0)$  for $t\in[50,500]$. 
The left-down panel shows $V(t,x)$ for fixed $t=200$. The right-down panel shows time evolution of $V(t,x)$ for fixed $t=250$.
In the two last panels, we observe that oscillations propagate from right to left, 
with a small propagation from left to right in the right extreme side of the domain.} 
\label{fig6}
\end{center}
\end{figure}

\end{document}